\documentclass[11pt,a4paper]{article}
\usepackage[T1]{fontenc}
\usepackage[utf8]{inputenc}
\usepackage[english]{babel}
\usepackage[margin=2.6cm]{geometry}
\usepackage{amsmath,amssymb,amsthm,mathtools}
\usepackage{microtype}
\usepackage[hidelinks]{hyperref}
\newcommand{\E}{\mathbb E}
\newcommand{\PP}{\mathbb P}
\newcommand{\R}{\mathbb R}
\newcommand{\ind}{\mathbf 1}
\newcommand{\cB}{\mathcal B}
\newcommand{\KL}{D}

\newtheorem{theorem}{Theorem}[section]
\newtheorem{proposition}[theorem]{Proposition}
\newtheorem{lemma}[theorem]{Lemma}
\newtheorem{corollary}[theorem]{Corollary}
\theoremstyle{remark}

\numberwithin{equation}{section}
\title{Moment estimates for chaoses with regular moment growth}
\author{Rafa{\l} Meller\footnote{Institute of Mathematics, University of Warsaw, Banacha 2, 02-097 Warsaw, Poland.}}
\date{}

\begin{document}
\maketitle
\begin{abstract}
We give two-sided moment estimates for decoupled chaoses of any order
generated by independent symmetric random variables with regular moment
growth: $\|X\|_{2p}\le K\|X\|_p$ for every $p\ge1$.
The estimates involve only the coefficient tensor and the one-dimensional
tails. They use the partition norms of Adamczak and Lata{\l}a, with
constants depending only on the order and on $K$.
The proof starts with log-concave tails. Gaussian smoothing gives the
estimate for the expected norm needed for induction, and a product comparison
reduces the remaining distributions to this case.
The extra tensor indices introduced by the comparison are removed using
Rademacher moments. A comparison with products of log-concave-tail variables then extends the formula
to regular moment growth. We also give tail bounds and explicit
Rademacher and Weibull formulas for $r\ge1$, including fourth-order examples.
This article was developed with the assistance of GPT Astra-6.
\end{abstract}

\section{Introduction and main results}\label{sec:main}

Let $A=(a_{\mathbf i})=(a_{i_1,\ldots,i_d})$ be a finite real array and let
\begin{equation}\label{eq:chaos}
 S_A(X)=\sum_{\mathbf i}a_{\mathbf i}X_{i_1}^1\cdots X_{i_d}^d,
\end{equation}
where all variables $(X_i^j)_{i\le n,j\le d}$ are independent and symmetric.
The superscript $j$ labels an independent family: even when $i_j=i_k$,
the variables $X_{i_j}^j$ and $X_{i_k}^k$ are independent for $j\ne k$.
We seek a formula for $\|S_A(X)\|_p$ in terms of $A$ and the tails
of the individual variables, with constants independent of $n$ and $p$.
This article was developed with the assistance of GPT Astra-6, including
the exploration of arguments and preparation of the text.

Lata{\l}a \cite{Lat06} obtained such formulas for Gaussian chaoses of
arbitrary order. Adamczak and Lata{\l}a \cite{AL12} introduced the
tail-dependent partition norms used below. They proved the lower bound
in every order, the upper bound through order three, and arbitrary-order
estimates for exponential variables.
In \cite{Mel19}, quadratic chaoses were treated under the more general
condition $\|X\|_{2p}\le K\|X\|_p$.
We prove the upper bound for log-concave tails in every order and then
extend the same formula to this moment-growth condition.
The extension uses the product comparison and multiplier argument in
\cite[Lemmas 3.2, 3.3 and 3.9]{Mel19}.

\subsection{The norms in the moment formula}
Write $[d]=\{1,\ldots,d\}$ and let $\mathcal P_d$ be its partitions into
nonempty blocks. For $I\subseteq[d]$, write $\mathbf i_I=(i_j)_{j\in I}$.
We use the notation of \cite{AL12,Mel19}:
\begin{equation}\label{eq:modified-tail}
 N_i^j(t)=-\log\PP(|X_i^j|\ge |t|),\qquad
 \widehat N_i^j(t)=
 \begin{cases}t^2,&|t|\le1,\\N_i^j(t),&|t|>1.\end{cases}
\end{equation}
The quadratic part describes small fluctuations; beyond one, the
constraint retains the tail of the actual variable.
The corresponding norm of a vector is
\[
 \|b\|_{X^j,p}=\sup\left\{\sum_i b_i x_i:
                              \sum_i\widehat N_i^j(x_i)\le p\right\}.
\]

For a tensor, group the indices into blocks.
Within each block choose one index along which to apply the tail
constraint, and take a Euclidean norm in the remaining indices.
Sum over the choices of the distinguished indices.
More precisely, for $\mathcal J=\{I_1,\ldots,I_k\}\in\mathcal P_d$, define
\begin{equation}\label{eq:partition-norm}
 \begin{split}
 \|A\|_{\mathcal J,p}^{N}
 =\sum_{s_1\in I_1,\ldots,s_k\in I_k}
 \sup\bigg\{&\sum_{\mathbf i}a_{\mathbf i}
                         \prod_{l=1}^k x^l_{\mathbf i_{I_l}}:\\
 &\sum_{i_{s_l}}\widehat N_{i_{s_l}}^{s_l}
       \left(\left\|(x^l_{\mathbf i_{I_l}})
                    _{\mathbf i_{I_l\setminus\{s_l\}}}\right\|_2\right)
                    \le p,\quad l=1,\ldots,k\bigg\}.
 \end{split}
\end{equation}
This is precisely \cite[equation (4)]{AL12}.
For a singleton block the inner Euclidean norm is an absolute value.
The coefficient array need not be symmetric or vanish on diagonals,
since \eqref{eq:chaos} is decoupled.
The notation $\asymp_{d,K}$ means comparison in both directions by
positive constants depending only on $d,K$; we omit $K$ when it is absent.

\subsection{Two moment theorems}

\begin{theorem}[Log-concave tails]\label{thm:main}
Suppose that all $N_i^j$ are convex on $[0,\infty)$ and normalized by
\begin{equation}\label{eq:normalization}
                    \inf\{t\ge0:N_i^j(t)\ge1\}=1.
\end{equation}
Then, for every real array $A$, $d\ge1$ and $p\ge2$,
\begin{equation}\label{eq:general-natural-moment-formula}
 \|S_A(X)\|_p\asymp_d
                   \sum_{\mathcal J\in\mathcal P_d}\|A\|_{\mathcal J,p}^{N}.
\end{equation}
\end{theorem}

\begin{theorem}[Regular growth of moments]\label{thm:regular}
The conclusion of the preceding theorem also holds for independent
symmetric variables satisfying
\begin{equation}\label{eq:regular-growth}
 \|X_i^j\|_{2p}\le K\|X_i^j\|_p\quad(p\ge1),\qquad
                         \|X_i^j\|_2=1/e ,
\end{equation}
with constants depending only on $d,K$:
\begin{equation}\label{eq:regular-moments}
 \|S_A(X)\|_p\asymp_{d,K}
                   \sum_{\mathcal J\in\mathcal P_d}\|A\|_{\mathcal J,p}^{N},
 \qquad p\ge2.
\end{equation}
All moments are assumed finite. The functions $N_i^j$ in this formula
are the actual tail exponents of $X_i^j$ and need not be convex.
\end{theorem}

The normalizations in these two statements are deliberately different:
\eqref{eq:normalization} follows \cite{AL12}, while
$\|X\|_2=1/e$ follows \cite{Mel19}.
A variable satisfying the moment-growth condition may have a large atom
at zero, in which case the quantile in \eqref{eq:normalization} can vanish.
For arbitrary nondegenerate variables in Theorem~\ref{thm:regular}, apply
the formula to $X_i^j/(e\|X_i^j\|_2)$ and replace $a_{\mathbf i}$ by
$a_{\mathbf i}\prod_j e\|X_{i_j}^j\|_2$.
Thus no restriction on the individual scales is imposed.

The lower bound in Theorem~\ref{thm:main} is
\cite[Theorem 3.1]{AL12}. Its upper bound is proved in
Sections~\ref{sec:root-class}--\ref{sec:diagonal-collapse}.
The proof of Theorem~\ref{thm:regular} follows immediately afterwards,
in Section~\ref{sec:regular}.
In particular, the Rademacher estimate is established before it is used
in either product reduction.

\subsection{How to read the formula}

For $d=1$, the formula is the scalar estimate
$\|\sum_i a_iX_i\|_p\asymp\|(a_i)\|_{X,p}$.
For $d=2$, the two partitions are $\{\{1,2\}\}$ and
$\{\{1\},\{2\}\}$. Thus, under the assumptions of either theorem,
\[
 \begin{aligned}
 \left\|\sum_{i,j}a_{ij}X_i^1X_j^2\right\|_p
 \asymp{}&\|A\|_{X^1,X^2,p}\\
 &+\left\|\left(\sqrt{\sum_j a_{ij}^2}\right)_i\right\|_{X^1,p}
 +\left\|\left(\sqrt{\sum_i a_{ij}^2}\right)_j\right\|_{X^2,p}.
 \end{aligned}
\]
where
\[
 \|A\|_{X^1,X^2,p}
 =\sup\left\{\sum_{i,j}a_{ij}u_iv_j:
       \sum_i\widehat N_i^1(u_i)\le p,\
       \sum_j\widehat N_j^2(v_j)\le p\right\}.
\]
The constants have the dependence specified in the relevant theorem.
The two singleton blocks give the bilinear supremum.
The single block gives the norms of the vectors of row and column lengths,
by Euclidean duality. This is the quadratic formula of
\cite[Theorem 2.1]{Mel19}.
In higher order the partitions organize these same two operations:
grouping indices in Euclidean norm and retaining separate tail constraints.

\subsection{Rademacher and Weibull consequences}

For an array $x$ indexed by a block $I$ and $s\in I$, abbreviate its
vector of Euclidean row lengths by
\[
 R_{I,s}(x)=
 \left(\left\|(x_{\mathbf i_I})_{\mathbf i_{I\setminus\{s\}}}\right\|_2
                                                        \right)_{i_s}.
\]
The admissible set for a block $I$ and $s\in I$ is
\begin{equation}\label{eq:block-body}
 \cB^N_{I,s}(p)=\left\{x:
 \sum_{i_s}\widehat N_{i_s}^s
 \left(\left\|(x_{\mathbf i_I})_{\mathbf i_{I\setminus\{s\}}}\right\|_2\right)
 \le p\right\}.
\end{equation}
For a family of symmetric block sets $C=(C_{I,s}(p))$, the notation
$\|A\|_{\mathcal J,p}^{C}$ means the expression
\eqref{eq:partition-norm} with $x^l\in C_{I_l,s_l}(p)$ as its constraints.

\begin{corollary}[Rademacher moments]\label{thm:rademacher}
Let $(\varepsilon_i^j)$ be independent random signs. Set
\[
 C^\varepsilon_{I,s}(p)=
 \{x:\|x\|_2\le\sqrt p,\ \|R_{I,s}(x)\|_\infty\le1\}.
\]
For every $p\ge2$,
\begin{equation}\label{eq:bernoulli-moments}
 \|S_A(\varepsilon)\|_p\asymp_d
           \sum_{\mathcal J\in\mathcal P_d}\|A\|_{\mathcal J,p}^{\varepsilon},
\end{equation}
where the superscript $\varepsilon$ refers to these block sets.
\end{corollary}

Write $B_2^I$ for the Euclidean unit ball of arrays indexed by $I$, and
\[
 \|x\|_{r,2;s}=\|R_{I,s}(x)\|_r,\qquad
 B_{r,2}^{I,s}=\{x:\|x\|_{r,2;s}\le1\}\quad(r\ge1).
\]
We also recall the Gaussian partition norms
\[
 \|A\|_{\mathcal J}=
 \sup_{\|x^l\|_2\le1}
                 \sum_{\mathbf i}a_{\mathbf i}\prod_{l=1}^k
                                                    x^l_{\mathbf i_{I_l}}.
\]

\begin{corollary}[Weibull moments]\label{thm:weibull}
Let the independent symmetric variables in \eqref{eq:chaos} have tails
$\PP(|X_i^j|\ge t)=\exp(-t^r)$, $t\ge0$, for a fixed $1\le r<\infty$.
For $p\ge2$, set
\begin{equation}\label{eq:weibull-balls}
 C^{(r)}_{I,s}(p)=
 \begin{cases}
 \sqrt p B_2^I+p^{1/r}B_{r,2}^{I,s},&1\le r<2,\\
 \sqrt p B_2^I\cap p^{1/r}B_{r,2}^{I,s},&2\le r<\infty .
 \end{cases}
\end{equation}
Then, with constants uniform in $r\ge1$,
\begin{equation}\label{eq:weibull-moments}
 \|S_A(X)\|_p\asymp_d
          \sum_{\mathcal J\in\mathcal P_d}\|A\|_{\mathcal J,p}^{(r)},
\end{equation}
where the superscript $(r)$ refers to the sets
\eqref{eq:weibull-balls}.
\end{corollary}

For $1\le r\le2$, expanding the sums of balls separates all powers of
$p$ from mixed tensor norms. For $r>2$, the intersections must be retained;
row sorting still simplifies each individual block.
Section~\ref{sec:examples} gives the derivations and fourth-order versions.
At $r=2$ we recover
$\sum_{\mathcal J}p^{|\mathcal J|/2}\|A\|_{\mathcal J}$, the expression
in the Gaussian formula \cite{Lat06}. The endpoint $r=1$ agrees with
\cite[Example 2]{KL15}.

\subsection{Proof strategy}

Conditioning on the last family lowers the order of the chaos, but leaves
a random coefficient tensor. Its moment is controlled by its mean norm
and a deterministic term, using the linear estimate
\eqref{eq:strong-weak}. The main difficulty is therefore the mean of
this random norm. The proof resolves it in three steps.

First, for tails $e^{-N(t^2)}$, shift independent Gaussian parameters
to a chosen point of each block. The entropy cost is controlled by
the squared Euclidean lengths of the parameters. An expansion of the
second moment produces exactly the merged blocks already present in
the moment formula. This closes the induction and includes random signs.

Second, compare a variable with log-concave tails to a product of two
variables from this class. The comparison duplicates every tensor index.
Split each parameter block into small rows and normalized large rows.
The normalized parts satisfy the Rademacher constraints; the remaining
row lengths become scalar multipliers. Products of independent signs
are again independent signs, so Rademacher moments remove the duplicated
indices. The scalar multipliers satisfy the original tail constraints.

Finally, the construction in \cite[Lemma 3.2]{Mel19} replaces a variable with regular moment
growth by a comparable product of a number of factors depending only
on $K$. The same reduction applies, with one change: the combined
multipliers lie in the convex hull of the original parameter set.
That is sufficient because the expression is linear in each parameter.
The argument for these multipliers is given after the log-concave proof
in Section~\ref{sec:regular}.

\section{Tools for the proof}\label{sec:preliminaries}

For $\mathbf s=(s_l)_{l\le k}$ with $s_l\in I_l$, let
$\|A\|_{\mathcal J,p;\mathbf s}^{N}$ denote the corresponding summand
in \eqref{eq:partition-norm}. We use the same convention for block sets
$C$, including the Rademacher sets $C^\varepsilon$.
To shorten the repeated sums, write
\[
 \|A\|_p^N:=\sum_{\mathcal J\in\mathcal P_d}\|A\|_{\mathcal J,p}^N.
\]
All parameter sets considered below are compact and symmetric.
Suprema are measurable, and absolute values around the entire
multilinear expression do not change them.

In Sections~\ref{sec:preliminaries}--\ref{sec:diagonal-collapse}
we use the assumptions and normalization of Theorem~\ref{thm:main}.
Convexity and \eqref{eq:normalization} imply
\begin{equation}\label{eq:tail-basic}
 N(0)=0,\qquad N(1)\le1,\qquad N(t)\ge t\quad(t>1).
\end{equation}
If $N$ is finite at some point beyond one, continuity and normalization
give $N(1)=1$; otherwise the last inequality is interpreted in the
extended sense. In particular all normalized variables have bounded
second moments, since
\[
 \E|X|^2=\int_0^\infty 2t\,\PP(|X|>t)\,dt
       \le1+\int_1^\infty 2t e^{-t}\,dt.
\]

We use two results from \cite{AL12}. Its Theorem 3.1 gives
\begin{equation}\label{eq:known-lower}
 \|A\|^{N}_{p}\le L_d\|S_A(X)\|_p.
\end{equation}
Its Lemma 5.9, attributed there to \cite[Theorem 1]{Lat96}, says that
for a bounded set $T\subset\R^n$ and $p\ge2$,
\begin{equation}\label{eq:strong-weak}
 \left\|\sup_{t\in T}|\langle t,X\rangle|\right\|_p
 \le L\left[
       \E\sup_{t\in T}|\langle t,X\rangle|
       +\sup_{\substack{t\in T\\\sum_i\widehat N_i(z_i)\le p}}
                                      \langle t,z\rangle\right].
\end{equation}
The first term is the expected supremum, while the second is a
deterministic optimization over the tail constraint.
In the induction, the second term will add a singleton block.
The expected supremum is the term that still requires an estimate.

We also use the entropy inequality
\begin{equation}\label{eq:entropy}
 \int f\,d\rho\le \KL(\rho\Vert\mu)+\log\int e^f\,d\mu,\qquad
 \KL(\rho\Vert\mu)=\int\log\frac{d\rho}{d\mu}\,d\rho ,
\end{equation}
for $\rho\ll\mu$ and the integrable expressions under consideration.
See \cite[Lemma A.1]{AA25}; its use with Gaussian shifts in tensor
problems is developed in the same paper. We verify the needed
integrability in the application below.

We will apply \eqref{eq:entropy} to an expression multiplied by a scale
$\lambda$. The next lemma chooses this scale from a deterministic norm
bound: if that bound is $M$, then $\lambda=p/M$ keeps the sum of
conjugate functions at most $p$.

\begin{lemma}\label{lem:dual-budget}
Let $\phi_i$ be nonnegative even convex functions with $\phi_i(0)=0$,
possibly taking the value $+\infty$. Put
\[
 F(u)=\sum_i\phi_i(u_i),\qquad
 K_p(v)=\sup_{F(u)\le p}\langle v,u\rangle,\qquad
 \phi_i^*(t)=\sup_{u\in\R}\{tu-\phi_i(u)\}.
\]
If $K_p(v)\le M$ and $M>0$, then
\begin{equation}\label{eq:dual-budget}
                  \sum_i\phi_i^*(pv_i/M)\le p.
\end{equation}
\end{lemma}
\begin{proof}
If $F(u)\le p$, then $(p/M)\langle v,u\rangle-F(u)\le p$.
If $p<F(u)<\infty$, convexity gives $F(pu/F(u))\le p$, and consequently
$(p/M)\langle v,u\rangle\le F(u)$.
Taking the supremum over $u$ with finite $F(u)$ proves the claim.
\end{proof}

\section{Gaussian smoothing}
\label{sec:root-class}

In this section $N_i^j$ are normalized convex tail exponents as above,
and $Y_i^j$ are independent symmetric variables with
\[
 \PP(|Y_i^j|\ge t)=e^{-Q_i^j(t)},\qquad Q_i^j(t)=N_i^j(t^2).
\]
The functions $Q_i^j$ are again convex and normalized.
Their modified exponents dominate $t^2$, so an admissible block has
squared Euclidean norm at most $p$. This is the property that controls
the cost of a Gaussian shift.
We prove the moment formula for this class directly.
Random signs are included by taking
$N(t)=0$ for $t\le1$ and $N(t)=+\infty$ for $t>1$.

\subsection{A logarithmic moment generating function estimate}

For one of these variables, suppress the indices and set
\begin{equation}\label{eq:root-potential}
 \phi(t)=\widehat Q(t)=F(t^2),\qquad
 F(s)=\begin{cases}s,&0\le s\le1,\\N(s),&s>1.\end{cases}
\end{equation}
The function $F$ is closed, convex, and satisfies $F(s)\ge s$.
Indeed, if $N$ is finite beyond one, then $N(1)=1$ and all its
right slopes there are at least one. If not, $F$ is the identity on
$[0,1]$ and infinite outside this interval. In particular
$\phi$ is convex and $\phi(t)\ge t^2$.

After Gaussian smoothing, a coefficient becomes a random polynomial $H$.
The next estimate bounds its averaged logarithmic moment generating
function using only $\|H\|_2$.

\begin{lemma}\label{lem:root-mgf}
Let $\psi(t)=\log\E e^{tY}$. For every $t\in\R$, real random variable
$H\in L_2$ and $\lambda\ge0$,
\begin{equation}\label{eq:root-mgf-l2}
 \psi(t)\le C\phi^*(t),\qquad
 \E\psi(\lambda H)\le C\phi^*(\lambda\|H\|_2),
\end{equation}
where $C$ is universal.
\end{lemma}
\begin{proof}
The key point is that $q\mapsto\phi^*(\sqrt q)$ is concave.
This will allow Jensen's inequality to replace $H$ by its $L_2$ norm.
To prove concavity, we use the identity
\begin{equation}\label{eq:root-tail-duality}
 \phi^*(t)=\inf_{a\ge1}
       \left\{\frac{t^2}{4a}+F^*(a)\right\},
 \qquad F^*(a)=\sup_{s\ge0}\{as-F(s)\}.
\end{equation}
The upper bound follows from $F(u^2)\ge au^2-F^*(a)$.
For $t>0$, the maximum of $tu-F(u^2)$ is attained at some $u>0$,
because $F(u^2)\ge u^2$ and $F(u^2)=u^2$ near zero.
The optimality condition gives $a\in\partial F(u^2)$ with $t=2au$.
Here $a\ge1$, and equality in convex duality gives
$F^*(a)=au^2-F(u^2)$. Substitution yields the opposite inequality.
At $t=0$, choose $a=1$; negative $t$ follow by symmetry.

Thus $q\mapsto\phi^*(\sqrt q)$ is concave: the right side of
\eqref{eq:root-tail-duality} is an infimum of affine functions of $q=t^2$.
Since this function vanishes at zero, we have
\begin{equation}\label{eq:root-jensen}
 \phi^*(2t)\le4\phi^*(t),\qquad
 \E\phi^*(\lambda H)\le\phi^*(\lambda\|H\|_2).
\end{equation}
Also $\phi^*(t)\le t^2/4$, with equality for $|t|\le2$.

The tail definition implies
$\PP(\phi(Y)>s)\le e^{-s}$ for $s\ge1$, including at possible atoms.
Consequently
\[
 \E e^{\phi(Y)/2}
 =1+\frac12\int_0^\infty e^{s/2}\PP(\phi(Y)>s)\,ds
 \le e^{1/2}+e^{-1/2}<3.
\]
Young's inequality now gives
$\psi(t)\le\log3+\phi^*(2t)/2$.
For $|t|\ge1$, combine this with \eqref{eq:root-jensen} and
$\phi^*(t)\ge1/4$.
For $|t|\le1$, symmetry and the power series of the hyperbolic cosine give
\[
 \psi(t)\le\E(\cosh(tY)-1)
          \le t^2(\E\cosh Y-1)\le Ct^2=4C\phi^*(t).
\]
The bound on $\E\cosh Y=\E e^Y$ follows from the preceding estimate
at $t=1$ and $\phi^*(2)=1$. This proves the first part of
\eqref{eq:root-mgf-l2}; the second follows from \eqref{eq:root-jensen}.
\end{proof}

\subsection{The expected norm after conditioning}

We now estimate the random coefficient tensor produced by conditioning
on one family. Its order is one less than that of the original tensor,
but the bound will involve only the original partition norms.

Let $m\ge2$ and contract the last index of a tensor $A$:
\[
 B_{\mathbf i_{[m-1]}}
             =\sum_{\ell=1}^n a_{\mathbf i_{[m-1]},\ell}Y_\ell^m.
\]
For a partition with distinguished indices $(\mathcal J,\mathbf s)$ of $[m-1]$ and
$\mathcal R\subseteq\mathcal J$, define
\[
 J_{\mathcal R}=\{m\}\cup\bigcup_{I\in\mathcal R}I,\qquad
 \mathcal J_{\mathcal R}
          =(\mathcal J\setminus\mathcal R)\cup\{J_{\mathcal R}\}.
\]
In $J_{\mathcal R}$ choose distinguished index $m$ and retain all
other distinguished indices; denote the resulting choice by
$\mathbf s_{\mathcal R}$.

\begin{proposition}\label{prop:root-contraction}
Writing $k=|\mathcal J|$, we have
\begin{equation}\label{eq:root-contraction}
 \E\|B\|^{Q}_{\mathcal J,p;\mathbf s}
 \le(C+k/2)
        \sum_{\mathcal R\subseteq\mathcal J}
          \|A\|^{Q}_{\mathcal J_{\mathcal R},p;\mathbf s_{\mathcal R}}.
\end{equation}
In particular,
\begin{equation}\label{eq:expected-contraction}
                 \E \|B\|^{Q}_{p}\le C_m \|A\|^{Q}_{p}.
\end{equation}
\end{proposition}
\begin{proof}
\noindent\textbf{1. Shift Gaussian parameters to the chosen blocks.}
Let $\mu$ be the product law of independent standard Gaussian arrays
$G^I$, one array for each block $I\in\mathcal J$.
For admissible parameters $x=(x^I)_{I\in\mathcal J}$, let $\rho_x$ be the law
of $(x^I+G^I)_{I\in\mathcal J}$. Since $\widehat Q(t)\ge t^2$,
\begin{equation}\label{eq:shift-cost}
       \KL(\rho_x\Vert\mu)=\frac12\sum_{I\in\mathcal J}\|x^I\|_2^2\le kp/2.
\end{equation}
Define
\[
 h_\ell(\theta)=
       \sum_{\mathbf i_{[m-1]}}a_{\mathbf i_{[m-1]},\ell}
                          \prod_{I\in\mathcal J}\theta^I_{\mathbf i_I},
 \qquad v_\ell(x)^2=\E_{\rho_x}h_\ell^2.
\]
Independence and centering give $\E_{\rho_x}h_\ell=h_\ell(x)$.
Thus averaging preserves the coefficient at the chosen point.

\par\smallskip\noindent\textbf{2. Identify the norms created by the second moment.}
In each block expand $x^I+G^I$, choosing either $x^I$ or $G^I$.
The subfamily $\mathcal R$ records the blocks in which the Gaussian
term was selected. This gives
\begin{equation}\label{eq:variance-expansion}
 v_\ell(x)^2=\sum_{\mathcal R\subseteq\mathcal J}c_{\mathcal R,\ell}(x)^2,
\end{equation}
where, with $U=\bigcup_{I\in\mathcal R}I$,
\[
 c_{\mathcal R,\ell}(x)=
 \left(\sum_{\mathbf i_U}
   \left[\sum_{\mathbf i_{[m-1]\setminus U}}
      a_{\mathbf i_{[m-1]},\ell}
            \prod_{I\in\mathcal J\setminus\mathcal R}x^I_{\mathbf i_I}
   \right]^2\right)^{1/2}.
\]
Terms with different $\mathcal R$ are orthogonal, because some independent
centered Gaussian block occurs in just one of them.
For fixed $\mathcal R$, the covariances of the Gaussian coordinates
are Kronecker deltas, which gives the displayed sum of squares.

Let $K_p$ in Lemma~\ref{lem:dual-budget} use
$\phi_\ell=\widehat Q_\ell^m$.
The Gaussian blocks contribute a Euclidean norm.
They can therefore be joined to the last index in the block
$J_{\mathcal R}$, with that last index distinguished.
Rowwise Euclidean duality gives
\begin{equation}\label{eq:merged-block}
 K_p(c_{\mathcal R}(x))
       \le\|A\|^{Q}_{\mathcal J_{\mathcal R},p;\mathbf s_{\mathcal R}}.
\end{equation}
To see this, take $u_\ell\ge0$ with $\sum_\ell\phi_\ell(u_\ell)\le p$.
In the new block $J_{\mathcal R}$ choose the row indexed by $\ell$
in the direction of the vector inside the brackets defining
$c_{\mathcal R,\ell}$, with length $u_\ell$; use zero when this vector
vanishes. This is an admissible new block and gives the contraction
$\sum_\ell u_\ell c_{\mathcal R,\ell}(x)$.
Nonnegative $u$ suffice, since the constraint is invariant under
coordinatewise changes of signs.

\par\smallskip\noindent\textbf{3. Choose one scale for all admissible parameters.}
Denote the sum on the right of \eqref{eq:root-contraction} by $M_0$.
Since $v(x)\le\sum_{\mathcal R}c_{\mathcal R}(x)$ coordinatewise,
monotonicity and subadditivity of $K_p$ imply $K_p(v(x))\le M_0$.
If $M_0=0$, the term $\mathcal R=\mathcal J$ forces $A=0$.
Otherwise set $\lambda=p/M_0$. This choice depends only on $A,p$
and the partition, not on $x$.
Lemmas~\ref{lem:dual-budget} and \ref{lem:root-mgf} then give,
uniformly in admissible $x$,
\begin{equation}\label{eq:mgf-penalty}
 \E_{\rho_x}\sum_\ell\psi_\ell(\lambda h_\ell)
 \le C\sum_\ell\phi_\ell^*(\lambda v_\ell(x))\le Cp,
 \qquad \psi_\ell(t)=\log\E e^{tY_\ell^m}.
\end{equation}

\par\smallskip\noindent\textbf{4. Apply the entropy inequality before taking the supremum.}
On the original probability space of $Y^m$, set
\[
 Z(Y^m)=\int\exp\left\{\lambda\sum_\ell Y_\ell^m h_\ell(\theta)
                  -\sum_\ell\psi_\ell(\lambda h_\ell(\theta))\right\}
                                                    \,d\mu(\theta).
\]
Independence and Tonelli's theorem give $\E Z=1$, so $0<Z<\infty$
almost surely. Moreover $\E_\mu h_\ell=0$ and $\psi_\ell(t)\le Ct^2$,
while $h_\ell$ is a Gaussian polynomial. Jensen's inequality therefore
gives the finite deterministic lower bound
\[
 \log Z\ge-\sum_\ell\E_\mu\psi_\ell(\lambda h_\ell).
\]
Also $(\log Z)_+\le Z$. Thus $\log Z$ is integrable and
$\E\log Z\le0$.

Apply \eqref{eq:entropy} to the exponent defining $Z$, with
$(\rho,\mu)=(\rho_x,\mu)$. Its integral under $\rho_x$ is finite,
by the polynomial form of $h_\ell$ and the quadratic upper bound on
$\psi_\ell$. Equations \eqref{eq:shift-cost} and \eqref{eq:mgf-penalty}
give, on the single event $\{Z<\infty\}$, for every admissible $x$,
\[
 \lambda\sum_\ell Y_\ell^m h_\ell(x)
                     \le(k/2+C)p+\log Z(Y^m).
\]
The right-hand side no longer depends on $x$.
We may therefore take the supremum over all admissible blocks on the left.
Taking expectation uses $\E\log Z\le0$; dividing by $\lambda=p/M_0$
proves \eqref{eq:root-contraction}.
Summing over the finitely many partitions and distinguished indices
proves \eqref{eq:expected-contraction}; their number depends only on $m$.
\end{proof}

\subsection{From the mean estimate to moments}

The preceding proposition controls the mean of the random norm.
The linear estimate \eqref{eq:strong-weak} supplies its $p$th moment.
We give the induction explicitly, since the Rademacher case obtained
here is needed in both later reductions.

\begin{theorem}\label{thm:root-moments}
For the independent variables with tails $e^{-N_i^j(t^2)}$,
\begin{equation}\label{eq:root-tail-moments}
                    \|S_A(Y)\|_p\asymp_d \|A\|^{Q}_{p},\qquad p\ge2.
\end{equation}
In particular Corollary~\ref{thm:rademacher} holds in every order.
\end{theorem}
\begin{proof}
The lower bound is \eqref{eq:known-lower}. For the upper bound in order
one, apply \eqref{eq:strong-weak} to a single coefficient vector.
The mean is at most $C\|A\|_2$ by independence, symmetry and the uniform
second-moment bound. The unit Euclidean ball is contained in
$\cB^Q_{\{1\},1}(p)$, so this term is bounded by $C\|A\|_p^Q$.

Suppose the upper bound holds in order $m-1$, and condition on $Y^m$.
The induction hypothesis and Minkowski's inequality give
\[
 \|S_A(Y)\|_p
       \le C_{m-1}\sum_{\mathcal J,\mathbf s}
                    \big\|\|B\|^{Q}_{\mathcal J,p;\mathbf s}\big\|_p.
\]
For a fixed partition and choice of distinguished indices, the coefficient set
$\{(h_\ell(x))_\ell:x^I\in\cB^Q_{I,s_I}(p)\}$ is compact and symmetric.
By \eqref{eq:strong-weak},
\[
 \big\|\|B\|^{Q}_{\mathcal J,p;\mathbf s}\big\|_p
 \le L\left[\E\|B\|^{Q}_{\mathcal J,p;\mathbf s}
       +\|A\|^{Q}_{\mathcal J\cup\{\{m\}\},p;(\mathbf s,m)}\right].
\]
The second term is exactly the weak term in \eqref{eq:strong-weak}:
it adds a vector with $\sum_\ell\widehat Q_\ell^m(z_\ell)\le p$.
Summing and using \eqref{eq:expected-contraction} yields
\[
 \|S_A(Y)\|_p
 \le LC_{m-1}\big[\E \|B\|^{Q}_{p}+\|A\|^{Q}_{p}\big]
 \le C_m \|A\|^{Q}_{p}.
\]
This completes the induction independently of the general-tail upper bound.
\end{proof}

\section{From squared-argument tails to log-concave tails}
\label{sec:product-reduction}

The preceding proof applies to tails $e^{-N(t^2)}$.
To handle a general tail $e^{-N(t)}$, we compare its variable to a product
of two variables from the class just treated.
The price is a tensor of twice the order.

Return to the variables $X_i^j$ with general normalized convex
exponents $N_i^j$. For every $(i,j)$, take two independent symmetric
variables $Y_i^{j,0},Y_i^{j,1}$ with absolute tails
$\exp(-N_i^j(t^2))$, all these variables being mutually independent.

\begin{lemma}\label{lem:product}
For every real tensor $A$ and $p\ge1$,
\begin{equation}\label{eq:product-comparison}
 \|S_A(X)\|_p
 \le2^d\left\|
       \sum_{\mathbf i}a_{\mathbf i}
            \prod_{j=1}^d Y_{i_j}^{j,0}Y_{i_j}^{j,1}
                                      \right\|_p.
\end{equation}
\end{lemma}
\begin{proof}
For one pair, independence gives
\[
 \PP(|Y^0Y^1|\ge t)
 \ge\PP(|Y^0|\ge\sqrt t,\ |Y^1|\ge\sqrt t)
 =e^{-2N(t)}\ge e^{-N(2t)}
 =\PP(|X|\ge2t),
\]
using $N(2t)\ge2N(t)$.
Thus $|X|\le_{\mathrm{st}}2|Y^0Y^1|$.
Couple these nonnegative variables in this order separately for each
$(i,j)$, and realize symmetric variables by independent signs times
their absolute values. Conditional on the coupled absolute values,
the contraction principle for signs bounds the smaller coefficients
by the larger ones in $L_p$.

More explicitly, if $0\le a_i\le b_i$, auxiliary independent signs
with means $a_i/b_i$ express the sum with coefficients $a_i$
as a conditional average of sums with coefficients $b_i$.
Jensen's inequality and invariance of independent uniform signs under
multiplication by independent signs give the contraction; zero $b_i$
may be omitted. Apply this argument to one family at a time, conditioning
on the other families. The factor two is paid once per family, giving
$2^d$, rather than once per coordinate.
The products of the two independent signs in $Y^0Y^1$ are again
independent uniform signs. This proves \eqref{eq:product-comparison}.
\end{proof}

Let $\Omega=[d]\times\{0,1\}$ and define the tensor with duplicated
indices by
\begin{equation}\label{eq:diagonal-lift}
 \widetilde a_{(i_{j,a})_{(j,a)\in\Omega}}
 =a_{i_{1,0},\ldots,i_{d,0}}
                     \prod_{j=1}^d\ind_{\{i_{j,0}=i_{j,1}\}}.
\end{equation}
For either copy of coordinate $j$, use the tail exponent
$Q_i^{j,a}(t)=N_i^j(t^2)$.
Theorem~\ref{thm:root-moments}, now in order $2d$, and
Lemma~\ref{lem:product} show that
\begin{equation}\label{eq:lifted-upper}
                    \|S_A(X)\|_p\le C_d \|\widetilde A\|^{Q}_{p}.
\end{equation}
To prove Theorem~\ref{thm:main}, it therefore remains to compare this
last deterministic expression with $\|A\|^{N}_{p}$.
Partitions of $\Omega$ can cross the pairs $\{(j,0),(j,1)\}$, so
collapsing their blocks directly does not preserve their row constraints.
The next section avoids that obstruction.

\section{Removing duplicated indices}\label{sec:diagonal-collapse}

\begin{proposition}\label{prop:diagonal-collapse}
For every normalized family $N_i^j$, every real tensor $A$ and $p\ge2$,
\begin{equation}\label{eq:diagonal-collapse-proved}
                  \|\widetilde A\|^{Q}_{p}\le C_d \|A\|^{N}_{p}.
\end{equation}
\end{proposition}

A partition of the duplicated indices need not descend to a partition
of the original indices. We instead separate the row lengths from
the row directions. After normalization, the directions are admissible
for random signs. The lengths can be combined into original singleton
constraints, using the multiplier idea in
\cite[equations (17)--(19), preprint version]{Mel16}.
The proof below implements these two operations without matching
the lifted partition to an original one.

\begin{proof}
Fix a partition $\mathcal Q$ of $\Omega$, distinguished indices
$s_B=(j_B,a_B)\in B$, and admissible parameters
$x^B\in\cB^Q_{B,s_B}(p)$ for $B\in\mathcal Q$.
Let $r_i^B$ be the Euclidean norm of row $i$ of $x^B$ along $s_B$.
Its constraint is
\[
 \sum_i\phi_i^{j_B}(r_i^B)\le p,\qquad
 \phi_i^j(t)=
 \begin{cases}t^2,&0\le t\le1,\\N_i^j(t^2),&t>1.\end{cases}
\]

\par\smallskip\noindent\textbf{1. Separate row lengths from row directions.}
The cutoff is one because that is the Rademacher row constraint.
For each row set
\[
 u_i^B=r_i^B\ind_{\{r_i^B>1\}},\qquad
 v^{B,0}_{\mathbf i_B}
       =x^B_{\mathbf i_B}\ind_{\{r_{i_{s_B}}^B\le1\}},
 \qquad
 v^{B,1}_{\mathbf i_B}=
 \begin{cases}
 x^B_{\mathbf i_B}/r_{i_{s_B}}^B,&r_{i_{s_B}}^B>1,\\
 0,&r_{i_{s_B}}^B\le1 .
 \end{cases}
\]
Then $x^B=v^{B,0}+u^Bv^{B,1}$, with $u^B$ acting only in
the distinguished coordinate. Both $v^{B,0}$ and $v^{B,1}$ belong
to $C^\varepsilon_{B,s_B}(p)$: their row norms are at most one,
and their squared Euclidean norms are at most $p$.
For $v^{B,1}$, the latter assertion follows from
\[
 \#\{i:r_i^B>1\}\le
                 \sum_{i:r_i^B>1}N_i^{j_B}((r_i^B)^2)\le p.
\]
The multipliers satisfy
\begin{equation}\label{eq:row-multipliers}
 u_i^B\in\{0\}\cup(1,\infty),\qquad
                    \sum_iN_i^{j_B}((u_i^B)^2)\le p.
\end{equation}

\par\smallskip\noindent\textbf{2. Check the tail constraint for each combined length.}
Expand the product over $B\in\mathcal Q$. A term selects
$u^Bv^{B,1}$ from some blocks $\mathcal S\subseteq\mathcal Q$
and $v^{B,0}$ from the others.
On the support of $\widetilde A$, the indices $i_{j,0}$ and $i_{j,1}$
are equal. Combine all selected multipliers with original coordinate
$j$ into $w_i^j$. There are at most two such multipliers: each of the
two copies of $j$ belongs to one block and can be distinguished only there.
Let $J\subseteq[d]$ consist of coordinates with at least one multiplier;
for the other coordinates set $w_i^j=1$.

For every $j\in J$,
\begin{equation}\label{eq:collapsed-budget}
                         \sum_i\widehat N_i^j(w_i^j)\le p.
\end{equation}
With one multiplier, this follows from $N_i^j(u_i)\le N_i^j(u_i^2)$.
With two, convexity and monotonicity give
\[
 N_i^j(u_iv_i)
 \le N_i^j((u_i^2+v_i^2)/2)
 \le\tfrac12N_i^j(u_i^2)+\tfrac12N_i^j(v_i^2).
\]
Every product is zero or greater than one, so
$\widehat N_i^j(w_i^j)=N_i^j(w_i^j)$, also at zero.
Now sum and use \eqref{eq:row-multipliers}.
This is an exact budget bound; no factor depending on $p$ is lost.

\par\smallskip\noindent\textbf{3. Remove the extra indices using random signs.}
Set $a^w_{\mathbf i}=a_{\mathbf i}\prod_{j\in J}w_{i_j}^j$.
The chosen term is a contraction of $\widetilde A^w$ against
Rademacher-admissible blocks $v^{B,0}$ or $v^{B,1}$.
The known lower bound \eqref{eq:known-lower}, in order $2d$, gives
\begin{align}\label{eq:bernoulli-collapse}
 |\text{term}|
 &\le\|\widetilde A^w\|^{\varepsilon}_{\mathcal Q,p;\mathbf s}
 \le L_{2d}
       \left\|\sum_{\mathbf i}a^w_{\mathbf i}
            \prod_{j=1}^d\varepsilon_{i_j}^{j,0}
                              \varepsilon_{i_j}^{j,1}\right\|_p
   =L_{2d}\|S_{A^w}(\eta)\|_p .
\end{align}
Here the superscript $\varepsilon$ denotes the functional with
$C^\varepsilon(p)$, and
$\eta_i^j=\varepsilon_i^{j,0}\varepsilon_i^{j,1}$
are independent Rademacher variables.
The lower bound applies to arbitrary real decoupled tensors, including
$\widetilde A^w$; neither its diagonal support nor its coefficient signs
require a change.

\par\smallskip\noindent\textbf{4. Absorb the multipliers into singleton blocks.}
Put $L=[d]\setminus J$. For fixed signs
$\sigma^j\in\{-1,1\}^n$, $j\in J$, define
\[
 C^{w,\sigma}_{\mathbf i_L}
      =\sum_{\mathbf i_J}a_{\mathbf i}
                      \prod_{j\in J}w_{i_j}^j\sigma_{i_j}^j.
\]
Condition on the signs in the coordinates carrying multipliers.
The conditional $p$th moment is at most its maximum over these signs,
so
\begin{equation}\label{eq:freezing}
 \|S_{A^w}(\eta)\|_p
       \le\max_{\sigma^J}\|S_{C^{w,\sigma}}(\eta^L)\|_p
       \le C_d \|A\|^{N}_{p}.
\end{equation}
For the second inequality, suppose first that $L\ne\varnothing$ and
apply the already proved Rademacher upper bound in order $|L|$.
By \eqref{eq:collapsed-budget}, each $w^j\sigma^j$ is admissible
in the original singleton body $\cB^N_{\{j\},j}(p)$.
Moreover,
\[
 C^\varepsilon_{I,s}(p)\subseteq\cB^N_{I,s}(p),
\]
because all its row lengths are at most one, where every modified
tail exponent equals the square. Consequently, for a
partition with distinguished indices $(\mathcal J,\mathbf t)$ of $L$,
\[
 \|C^{w,\sigma}\|^{\varepsilon}_{\mathcal J,p;\mathbf t}
 \le\|A\|^{N}_{\mathcal J\cup\{\{j\}:j\in J\},p;(\mathbf t,(j)_{j\in J})}.
\]
Their sum is a subsum of $\|A\|^{N}_{p}$, uniformly in $w,\sigma$.
If $L=\varnothing$, $C^{w,\sigma}$ is a scalar and is bounded in
absolute value by the all-singleton functional of $A$.
This proves \eqref{eq:freezing}.
The maximum over signs costs no cardinality factor, since every
choice is covered by the same deterministic budget.

There are $2^{|\mathcal Q|}\le2^{2d}$ expanded terms.
Equations \eqref{eq:bernoulli-collapse} and \eqref{eq:freezing}
bound each one uniformly in the original parameters $x^B$.
Take their supremum and sum over $\mathcal Q$ and its distinguished
indices. The number of these choices depends only on $d$, which
proves \eqref{eq:diagonal-collapse-proved}.
\end{proof}

\begin{proof}[Proof of Theorem~\ref{thm:main}]
Combine \eqref{eq:lifted-upper} with
Proposition~\ref{prop:diagonal-collapse} for the upper bound.
The lower bound is \eqref{eq:known-lower}.
\end{proof}

\section{Extension to regular moment growth}\label{sec:regular}

In this section assume \eqref{eq:regular-growth}; convexity of $N_i^j$
is no longer required. Fix
\[
 m=\max\{1,\lceil\log_2 K\rceil\},
\]
so that the variables belong to the class $\mathbb S(m)$ defined in \cite{Mel19}.
All constants in this section may depend on $d,m$, equivalently on $d,K$.
Chebyshev's inequality gives $N_i^j(t)\ge2+2\log t$ for $t>1$.
Thus the block sets remain compact and coordinatewise solid:
decreasing the absolute value of any entry preserves admissibility.
Their convex hulls have the same support functions.
This matters here because products of several row lengths will be
controlled in a convex hull, whereas the two-factor argument above
gave a point in the original set.

We first reduce to products of log-concave-tail variables, then control
their combined row lengths, and finally apply the reduction with random
signs. The lower bound is checked at the end of the section.

\subsection{Comparison with products}

By \cite[Lemma 3.2]{Mel19}, for each original coordinate there are
independent identically distributed symmetric variables
$Y_i^{j,1},\ldots,Y_i^{j,m}$ with log-concave tails.
Different original coordinates are assigned independent groups.
Writing
\[
 M_i^j(t)=-\log\PP(|Y_i^{j,1}|\ge t),
\]
these factors have the quantile normalization
\eqref{eq:normalization} and satisfy, with $C_0,t_0$ depending only on $m$,
\begin{equation}\label{eq:meller-tail-comparison}
 M_i^j(t)\le N_i^j(t^m)\le M_i^j(C_0t),\qquad t\ge t_0,
 \qquad c_m\le\E|Y_i^{j,1}|\le C_m .
\end{equation}
They can be coupled to $X_i^j$ so that
\[
 |X_i^j|\le C_m\left(1+\prod_{\nu=1}^m|Y_i^{j,\nu}|\right),\qquad
 \prod_{\nu=1}^m|Y_i^{j,\nu}|\le C_m(1+|X_i^j|).
\]
These are comparisons, not identities in distribution.
Applying \cite[Lemma 3.3]{Mel19} conditionally to each of the $d$ families gives
\begin{equation}\label{eq:regular-product}
 \|S_A(X)\|_p\asymp_{d,K}
 \left\|\sum_{\mathbf i}a_{\mathbf i}
                \prod_{j=1}^d\prod_{\nu=1}^mY_{i_j}^{j,\nu}\right\|_p.
\end{equation}
The additive ones in the coupling are removed by contraction and Jensen:
the means of the absolute values of both $X_i^j$ and the products
are bounded below by positive constants depending only on $m$.
Applying the scalar comparison conditionally is legitimate because the
coefficient array at each step is independent of the family being replaced.

Lift $A$ to the array $A^{[m]}$ on $[d]\times[m]$ by setting
\[
 a^{[m]}_{(i_{j,\nu})}
 =a_{i_{1,1},\ldots,i_{d,1}}
       \prod_{j=1}^d\prod_{\nu=2}^m
                         \ind_{\{i_{j,\nu}=i_{j,1}\}} .
\]
Use $M_i^j$ on every copy of coordinate $j$.
The already proved log-concave theorem and \eqref{eq:regular-product} imply
\begin{equation}\label{eq:regular-lifted}
 \|S_A(X)\|_p\le C_{d,K}
           \sum_{\mathcal Q\in\mathcal P_{md}}\|A^{[m]}\|_{\mathcal Q,p}^{M}.
\end{equation}

\subsection{Combining up to $m$ multipliers}

Here is the precise scalar step, adapted from the proof of
\cite[Lemma 3.9, equations (44)--(46)]{Mel19}.
Fix one family and suppress its superscript $j$.
Suppose $1\le q\le m$, $u_i^\nu\in\{0\}\cup(1,\infty)$, and
\[
 \sum_iM_i(u_i^\nu)\le p,\qquad \nu=1,\ldots,q,\qquad
                         w_i=\prod_{\nu=1}^q u_i^\nu .
\]
Let
\[
 B_X(p)=\{z:\sum_i\widehat N_i(z_i)\le p\}.
\]
We claim that
\begin{equation}\label{eq:regular-multiplier}
                     w\in C_m\,\operatorname{conv}B_X(p).
\end{equation}
Choose $C_0,t_0\ge1$ from \eqref{eq:meller-tail-comparison} and put
$T=C_0t_0$. Above $T$ we can transfer the factor's tail constraint
back to the original variable. Below $T$, a bound on the number
of nonzero coordinates will suffice.

The support of $w$ has cardinality at most $p$,
since it is contained in the nonzero support of each $u^\nu$
and $M_i(u_i^\nu)\ge u_i^\nu>1$ there.
Hence $\ind_{\{w_i>0\}}$ belongs to $B_X(p)$.
Also, for each $\nu$,
\[
 z_i^\nu=(u_i^\nu/C_0)^m\ind_{\{u_i^\nu>T\}}
\]
belongs to $B_X(p)$. Indeed, on $u_i^\nu>T$, substitute
$t=u_i^\nu/C_0>t_0$ in \eqref{eq:meller-tail-comparison} to get
\[
 \widehat N_i(z_i^\nu)=N_i((u_i^\nu/C_0)^m)
                         \le M_i(u_i^\nu).
\]
Sum over $i$ to obtain the required budget.
On the support of $w$, the arithmetic--geometric mean inequality and
$q\le m$ give
\begin{equation}\label{eq:regular-multiplier-majorant}
 w_i\le\frac1q\sum_{\nu=1}^q(u_i^\nu)^q
 \le T^m\ind_{\{w_i>0\}}+\frac{C_0^m}{q}\sum_{\nu=1}^qz_i^\nu .
\end{equation}
Off this support the last bound is automatic.
The right-hand side lies in
$(T^m+C_0^m)\operatorname{conv}B_X(p)$.
This convex hull is unconditional and coordinatewise solid, proving
\eqref{eq:regular-multiplier}. In particular, for every real vector $b$,
\[
             \sum_i |b_i|w_i\le C_m\|b\|_{X,p}.
\]
This support-function estimate is the form used below.
It permits each combined multiplier in a singleton block, up to a
constant depending only on $m$.

\subsection{Reduction of the lifted partition norms}

Fix a partition $\mathcal Q$ of $[d]\times[m]$, distinguished indices
$s_B=(j_B,\nu_B)$, and admissible parameters $x^B$ for
$\|A^{[m]}\|_{\mathcal Q,p;\mathbf s}^{M}$.
If $r_i^B$ is the Euclidean length of row $i$ of $x^B$, split the block
as in Proposition~\ref{prop:diagonal-collapse}:
\[
 x^B=v^{B,0}+u^Bv^{B,1},\qquad
 u_i^B=r_i^B\ind_{\{r_i^B>1\}},
\]
where $v^{B,0}$ consists of the rows of length at most one and
$v^{B,1}$ normalizes each remaining row to length one.
Both $v^{B,0},v^{B,1}$ are Rademacher-admissible:
\[
 \|R_{B,s_B}(v^{B,a})\|_\infty\le1,\qquad
 \|v^{B,a}\|_2^2\le p,\qquad a=0,1 .
\]
For the second tensor use
$\#\{i:r_i^B>1\}\le\sum_i\widehat M_i^{j_B}(r_i^B)\le p$.
Furthermore $\sum_iM_i^{j_B}(u_i^B)\le p$.

Expand the product over the blocks.
On the diagonal support of $A^{[m]}$, at most $m$ selected multipliers
belong to any one original coordinate $j$.
Let their product be $w^j$, and let $J\subseteq[d]$ denote the coordinates
with at least one selected multiplier. Equation
\eqref{eq:regular-multiplier} gives
\[
               w^j\in C_m\operatorname{conv}B_{X^j}(p),\qquad j\in J.
\]
For $a^w_{\mathbf i}=a_{\mathbf i}\prod_{j\in J}w_{i_j}^j$,
each expanded term is at most
\[
 \| (A^w)^{[m]}\|_{\mathcal Q,p;\mathbf s}^{\varepsilon}
 \le C_{md}\left\|
       \sum_{\mathbf i}a^w_{\mathbf i}
         \prod_{j=1}^d\prod_{\nu=1}^m\varepsilon_{i_j}^{j,\nu}\right\|_p
 =C_{md}\|S_{A^w}(\eta)\|_p .
\]
We used only the known Rademacher lower bound in order $md$.
The products $\eta_i^j=\prod_\nu\varepsilon_i^{j,\nu}$ are independent
random signs, so all extra indices have disappeared.

Freeze $\eta^j=\sigma^j$ for $j\in J$ and apply the Rademacher upper
bound to the remaining families. Exactly as in
\eqref{eq:freezing}, their block parameters are admissible in the
original $N$-bodies. The singleton vectors $w^j\sigma^j$ now belong
to $C_m\operatorname{conv}B_{X^j}(p)$.
A multilinear supremum is unchanged if an individual parameter set is
replaced by its convex hull. Thus every choice of the frozen signs is
bounded by
\[
             C_{d,K}\sum_{\mathcal J\in\mathcal P_d}\|A\|_{\mathcal J,p}^N.
\]
If no random families remain, the all-singleton term gives the same bound.
There is no cost for the number of frozen sign choices.
The number of expanded terms, lifted partitions and distinguished-index
choices depends only on $md$. Taking suprema and summing therefore gives
\begin{equation}\label{eq:regular-collapse}
 \sum_{\mathcal Q\in\mathcal P_{md}}\|A^{[m]}\|_{\mathcal Q,p}^{M}
 \le C_{d,K}\sum_{\mathcal J\in\mathcal P_d}\|A\|_{\mathcal J,p}^N.
\end{equation}
Together with \eqref{eq:regular-lifted}, this proves the upper bound
in Theorem~\ref{thm:regular}. We finish by verifying the lower bound.

\subsection{The lower bound}

We record why the lower-bound argument of \cite[Section 4]{AL12}
extends to the present distributions.
The one-dimensional estimate in \cite[Theorem 3.1]{Mel19} states that
\begin{equation}\label{eq:regular-linear}
 \left\|\sum_i b_iX_i^j\right\|_p
                         \asymp_K\|b\|_{X^j,p},\qquad p\ge2 .
\end{equation}
Moreover $\E|X_i^j|\ge\|X_i^j\|_2/K=1/(eK)$.
For any coefficient array $B$ of order $l$, symmetry, Jensen's inequality
applied to the absolute values, and contraction for signs give
\[
 \E|S_B(X)|\ge(eK)^{-l}\E|S_B(\varepsilon)|
                         \ge c_{l,K}\|B\|_2 .
\]
For completeness, the last comparison follows by iterating the
fourth-moment Khintchine inequality:
$\|S_B(\varepsilon)\|_4\le3^{l/2}\|B\|_2$ and
$\|S_B(\varepsilon)\|_2=\|B\|_2$; interpolation gives
$\|S_B(\varepsilon)\|_1\ge3^{-l}\|B\|_2$.

Fix an index $s$ of a block $I$ and a coefficient array
$b_{\mathbf i_I}$, and write
\[
 Z_{i_s}=\sum_{\mathbf i_{I\setminus\{s\}}}
            b_{\mathbf i_I}\prod_{j\in I\setminus\{s\}}X_{i_j}^j.
\]
The preceding estimate gives
$\E|Z_{i_s}|\ge c_{d,K}
             \|(b_{\mathbf i_I})_{\mathbf i_{I\setminus\{s\}}}\|_2$.
Condition on the other families, use symmetry in $X^s$, and apply
Jensen's inequality in the vector $(|Z_{i_s}|)_{i_s}$ to obtain
\begin{align*}
 \left\|\sum_{i_s}X_{i_s}^s Z_{i_s}\right\|_p
 &\ge\left\|\sum_{i_s}X_{i_s}^s\,\E|Z_{i_s}|\right\|_p\\
 &\ge c_{d,K}
   \left\|\left(\left\|(b_{\mathbf i_I})
            _{\mathbf i_{I\setminus\{s\}}}\right\|_2\right)_{i_s}
                                                        \right\|_{X^s,p}.
\end{align*}
The second line uses contraction and \eqref{eq:regular-linear}.
By rowwise Euclidean duality, its last norm is
$\sup_{x\in\cB^N_{I,s}(p)}\sum_{\mathbf i_I}b_{\mathbf i_I}x_{\mathbf i_I}$.
Apply this inequality successively to the blocks of a partition,
conditioning on the remaining families at each step. For every fixed
choice of admissible block parameters this gives
\[
 \left|\sum_{\mathbf i}a_{\mathbf i}
                        \prod_{I\in\mathcal J}x^I_{\mathbf i_I}\right|
                                      \le C_{d,K}\|S_A(X)\|_p.
\]
Taking the supremum and summing the finitely many choices proves
\begin{equation}\label{eq:regular-lower}
             \sum_{\mathcal J\in\mathcal P_d}\|A\|_{\mathcal J,p}^N
                                      \le C_{d,K}\|S_A(X)\|_p .
\end{equation}
Only the scalar estimate and a lower comparison of the first and second
moments of fixed-order chaoses have been used.

Combining \eqref{eq:regular-lower} with the upper bound proves
Theorem~\ref{thm:regular}.

\section{Reading the Weibull and Rademacher formulas}
\label{sec:examples}

We first verify the parameter sets in Corollaries~\ref{thm:rademacher}
and \ref{thm:weibull}. For Rademacher variables, $N(t)=0$ on $[0,1]$
and $N(t)=+\infty$ for $t>1$. Hence $\cB^N_{I,s}(p)$ is exactly
$C^\varepsilon_{I,s}(p)$.
For a Weibull tail, the modified exponent is
\[
 \phi_r(t)=
 \begin{cases}t^2,&|t|\le1,\\ |t|^r,&|t|>1.\end{cases}
\]
Let $D^{(r)}_{I,s}(p)$ be the body defined by
$\sum_{i_s}\phi_r(R_{I,s}(x)_{i_s})\le p$.

If $1\le r\le2$, splitting the rows according to whether their norm
is at most one gives
\[
 D^{(r)}_{I,s}(p)\subseteq
          \sqrt p B_2^I+p^{1/r}B_{r,2}^{I,s}.
\]
Conversely each summand on the right is contained in $D^{(r)}_{I,s}(p)$,
because $\phi_r(t)\le t^2$ and $\phi_r(t)\le |t|^r$.
For a multilinear functional, expansion of the sum in each block
therefore changes its supremum by at most $2^{|\mathcal J|}$.
If $r\ge2$, $\phi_r(t)=\max(t^2,|t|^r)$, and
\begin{equation}\label{eq:intersection-comparison}
 D^{(r)}_{I,s}(p)\subseteq
       \sqrt p B_2^I\cap p^{1/r}B_{r,2}^{I,s}
       \subseteq \sqrt2\,D^{(r)}_{I,s}(p).
\end{equation}
Indeed, the sum of the maxima is at most the sum of the two budgets,
and $\phi_r(t/\sqrt2)\le\phi_r(t)/2$.
These comparisons and Theorem~\ref{thm:main} prove the Rademacher formula
and the Weibull formula
with constants uniform in $r$.

\subsection{Mixed norms when $1\le r\le2$}

Each block in \eqref{eq:weibull-balls} offers two choices:
a Euclidean ball with scale $p^{1/2}$, or a mixed ball with scale
$p^{1/r}$. If a partition has $k$ blocks and $h$ use the mixed ball,
the resulting power is $p^{(k-h)/2+h/r}$.
We now define the corresponding norms, which are independent of $p$.

For a partition $\mathcal J$ and a subfamily $\mathcal H\subseteq\mathcal J$,
choose $s_I\in I$ only for $I\in\mathcal H$, and define
\begin{equation}\label{eq:mixed-norm}
 \|A\|_{\mathcal J,\mathcal H,\mathbf s;r}
 =\sup\left\{\sum_{\mathbf i}a_{\mathbf i}
                          \prod_{I\in\mathcal J}x^I_{\mathbf i_I}:
 \begin{array}{ll}
 \|x^I\|_2\le1,&I\notin\mathcal H,\\
 \|x^I\|_{r,2;s_I}\le1,&I\in\mathcal H
 \end{array}\right\}.
\end{equation}
For $1\le k\le d$ and $0\le h\le k$, let
\[
 \|A\|_{k,h}^{(r)}=
 \sum_{\substack{\mathcal J\in\mathcal P_d\\|\mathcal J|=k}}
 \ \sum_{\substack{\mathcal H\subseteq\mathcal J\\|\mathcal H|=h}}
 \ \sum_{\mathbf s\in\prod_{I\in\mathcal H}I}
           \|A\|_{\mathcal J,\mathcal H,\mathbf s;r}.
\]
The inner sum has one term when $\mathcal H=\varnothing$.
These sums of mixed norms do not depend on $p$.

\begin{corollary}\label{cor:fourteen}
For the variables in Corollary~\ref{thm:weibull}, if $1\le r\le2$, then
\begin{equation}\label{eq:weibull-polynomial}
 \|S_A(X)\|_p\asymp_d
 \sum_{k=1}^d\sum_{h=0}^k
                  p^{(k-h)/2+h/r}\|A\|_{k,h}^{(r)}.
\end{equation}
\end{corollary}

\begin{proof}
For $1\le r<2$, expand each Minkowski sum in \eqref{eq:weibull-balls}.
At $r=2$, first replace $\sqrt p B_2^I$ by $2\sqrt p B_2^I$,
which changes the comparison by at most $2^d$, and use the same expansion.
The supremum is at most the sum of the suprema of all expanded terms
and is at least each such supremum, so these two quantities differ
by at most $2^d$. A block using the Euclidean ball contributes
$p^{1/2}$, and a block using the mixed ball contributes $p^{1/r}$.
The choices of distinguished indices in Euclidean blocks merely
repeat identical terms, with multiplicities depending only on $d$.
Removing these repetitions gives \eqref{eq:weibull-polynomial}.
\end{proof}

The two one-block terms already show the distinction.
Writing $q=r/(r-1)$, with $q=\infty$ at $r=1$, we have
\[
 \|A\|_{1,0}^{(r)}=\|A\|_2,\qquad
 \|A\|_{1,1}^{(r)}=\sum_{s=1}^d
                \left\|R_{[d],s}(A)\right\|_q.
\]
At the other extreme, all blocks are singletons and all constraints
are of $\ell_r$ type:
\[
 \|A\|_{d,d}^{(r)}=
 \sup_{\|x^j\|_r\le1}\sum_{\mathbf i}a_{\mathbf i}\prod_{j=1}^d x^j_{i_j}.
\]
\paragraph{Order four.}
There are $k+1$ choices for the number $h$ of mixed blocks in a
partition with $k$ blocks. For $k=1,2,3,4$, this gives fourteen
terms in \eqref{eq:weibull-polynomial}:
\begin{align}\label{eq:fourteen}
 \|S_A(X)\|_p\asymp_4{}&
 \sqrt p\,\|A\|_{1,0}^{(r)}+p^{1/r}\|A\|_{1,1}^{(r)}\nonumber\\
 &+p\|A\|_{2,0}^{(r)}+p^{1/2+1/r}\|A\|_{2,1}^{(r)}+p^{2/r}\|A\|_{2,2}^{(r)}\nonumber\\
 &+p^{3/2}\|A\|_{3,0}^{(r)}+p^{1+1/r}\|A\|_{3,1}^{(r)}\nonumber\\
 &\hspace{16mm}+p^{1/2+2/r}\|A\|_{3,2}^{(r)}+p^{3/r}\|A\|_{3,3}^{(r)}\nonumber\\
 &+p^2\|A\|_{4,0}^{(r)}+p^{3/2+1/r}\|A\|_{4,1}^{(r)}
       +p^{1+2/r}\|A\|_{4,2}^{(r)}\nonumber\\
 &\hspace{16mm}+p^{1/2+3/r}\|A\|_{4,3}^{(r)}+p^{4/r}\|A\|_{4,4}^{(r)}.
\end{align}

The middle expressions in \eqref{eq:fourteen} retain genuine mixed
tensor norms; \eqref{eq:fourteen} does not replace them by entrywise norms.
For order four the underlying partitions are exactly
\[
\begin{array}{c|c|c}
\text{block sizes}&\text{partitions}&
               \text{choices of all distinguished indices}\\ \hline
4&1&4\\
3+1&4&12\\
2+2&3&12\\
2+1+1&6&12\\
1+1+1+1&1&1
\end{array}
\]
Thus there are fifteen partitions and forty-one choices including the distinguished indices.

\subsection{Row sorting for $r>2$ and for random signs}

An intersection of balls does not split into a sum of tensor norms.
We can still simplify one block at a time: take Euclidean norms along
its rows, then sort the resulting vector.
The cutoff at $\lfloor p\rfloor$ coordinates separates the part
controlled by the $\ell_r$ constraint from the Euclidean remainder.

For a vector $b\in\R^n$, define
\[
 \|b\|_{r,p}=\sup\{\langle b,u\rangle:
                         \|u\|_2\le\sqrt p,\ \|u\|_r\le p^{1/r}\},
 \qquad 2\le r\le\infty,
\]
where $p^{1/\infty}=1$.
For $1\le r<2$, it is convenient to set instead
\[
 \|b\|_{r,p}=\sqrt p\,\|b\|_2+p^{1/r}\|b\|_{r/(r-1)},
\]
where $r/(r-1)=\infty$ at $r=1$.
These are precisely the support functions of the singleton sets
$C^{(r)}(p)$, including the Rademacher case $r=\infty$.

\begin{lemma}\label{lem:sorting}
Let $2\le r\le\infty$, $q=r/(r-1)$, with $q=1$ if $r=\infty$,
and $m=\min(n,\lfloor p\rfloor)$. Write
$b_1^*\ge\cdots\ge b_n^*$ for the decreasing rearrangement of $(|b_i|)$.
Then, with absolute constants,
\begin{equation}\label{eq:sorting}
 \|b\|_{r,p}\asymp
 p^{1/r}\left(\sum_{i\le m}(b_i^*)^q\right)^{1/q}
       +\sqrt p\left(\sum_{i>m}(b_i^*)^2\right)^{1/2}.
\end{equation}
\end{lemma}
\begin{proof}
Denote the two terms on the right by $T$ and $R$.
H\"older's inequality on the first $m$ coordinates and
Cauchy--Schwarz on the rest give $\|b\|_{r,p}\le T+R$.
An optimizer for $T$ on the $\ell_r$ ball, supported on those $m$
coordinates, has Euclidean norm at most
$m^{1/2-1/r}p^{1/r}\le\sqrt p$. Consequently $\|b\|_{r,p}\ge T$.

If $m=n$, this finishes the proof. Otherwise let
$a=(\sum_{i>m}(b_i^*)^2)^{1/2}$.
When $a\ge\sqrt p\,b_m^*$ and $a>0$, the tail vector
$u_i=\sqrt p\,b_i^*/a$ has all coordinates at most one.
It has squared Euclidean norm $p$ and $\ell_r$ budget at most $p$
(for $r=\infty$, its maximum is at most one). Thus $\|b\|_{r,p}\ge R$.
If $a<\sqrt p\,b_m^*$, then
$R<pb_m^*\le2p^{1/r}m^{1/q}b_m^*\le2T$,
because $p/m\le2$. The zero-tail case is immediate.
In all cases $T+R\le3\|b\|_{r,p}$.
\end{proof}

For signs, \eqref{eq:sorting} becomes the explicit head--tail expression
\[
 \|b\|_{\infty,p}\asymp
       \sum_{i\le m}b_i^*
       +\sqrt p\left(\sum_{i>m}(b_i^*)^2\right)^{1/2}.
\]
Thus the cutoff involves sorted row lengths when a tensor block is
eliminated, rather than sorted entries of the full tensor.

For $d=4$, use vertical bars to separate blocks, so that, for example,
$\|A\|^{(r)}_{123|4,p}$ denotes the norm of the partition
$\{\{1,2,3\},\{4\}\}$.
Let $\mathfrak P_{22}$ be the three unordered partitions of $[4]$ into
two pairs. Then \eqref{eq:weibull-moments} reads
\begin{align}\label{eq:five-types}
 \|S_A(X)\|_p\asymp_4{}&
 \|A\|^{(r)}_{1234,p}
 +\sum_{\ell=1}^4\|A\|^{(r)}_{([4]\setminus\{\ell\})|\ell,p}
 +\sum_{\{I,J\}\in\mathfrak P_{22}}\|A\|^{(r)}_{I|J,p}\nonumber\\
 &+\sum_{\substack{I\subset[4]\\|I|=2}}
              \|A\|^{(r)}_{\{I\}\cup\{\{j\}:j\notin I\},p}
 +\|A\|^{(r)}_{1|2|3|4,p}.
\end{align}
The fourth term uses the pair $I$ and the two singleton blocks in $I^c$.
For random signs replace $(r)$ by $\varepsilon$.

Each block can be eliminated by rowwise Euclidean duality and the
vector norm $\|b\|_{r,p}$ above. For example, one of the three summands
of $\|A\|^{(r)}_{123|4,p}$ is
\[
 \sup_{z\in C^{(r)}_{\{4\},4}(p)}
 \left\|\left(
       \left[\sum_{j,k}\left(\sum_l a_{ijkl}z_l\right)^2\right]^{1/2}
                        \right)_i\right\|_{r,p}.
\]
The other two distinguish $j$ and $k$ instead of $i$.
Substituting \eqref{eq:sorting} gives a sum of the largest row lengths
in $\ell_q$ and an $\ell_2$ remainder. This simplification retains the
intersection constraint and all tensor indices.

One cannot obtain the $r>2$ formula by simply adding separate Gaussian
and $\ell_r$ contributions as in \eqref{eq:fourteen}.
Already for a single nonzero coefficient and $p\ge2$, the intersection
support is $p^{1/r}$ whereas the sum support is
$\sqrt p+p^{1/r}$. Their ratio is unbounded as $p$ increases.
For signs the corresponding values are $1$ and $\sqrt p+1$.
The intersection constraint is therefore essential.

\section{Tail bounds and nondecoupled chaoses}

\subsection{Tail probabilities}
The moment formula specifies the size of a deviation with probability
of exponential order $e^{-p}$. Markov's inequality gives the upper tail;
comparison of the $p$th and $2p$th moments gives the matching lower scale.

\begin{corollary}\label{cor:tails}
Under the assumptions of Theorem~\ref{thm:main}, if $A\ne0$,
there are positive constants $c_d,C_d$ such that, for $p\ge2$,
\begin{equation}\label{eq:tail-bounds}
 \PP\{|S_A(X)|\ge C_d\|A\|_p^N\}\le e^{-p},
 \qquad
 \PP\{|S_A(X)|\ge c_d\|A\|_p^N\}\ge e^{-C_dp}.
\end{equation}
\end{corollary}
\begin{proof}
The upper bound follows from Markov's inequality and
Theorem~\ref{thm:main}, increasing its upper constant by a factor $e$.
For $a\ge1$, convexity of $N$ and \eqref{eq:tail-basic} imply
$\widehat N(t/a)\le\widehat N(t)/a$.
When $t>1$ and $t/a\le1$, for instance,
$(t/a)^2\le t/a\le N(t)/a$; the other cases follow from the
quadratic formula or convexity. Rescaling each block therefore gives
\begin{equation}\label{eq:moment-scaling}
                  \|A\|^{N}_{ap}\le a^d\|A\|_p^N.
\end{equation}
Thus $\|S_A(X)\|_{2p}\le K_d\|S_A(X)\|_p$.
Paley--Zygmund applied to $|S_A(X)|^p$ yields
\[
 \PP\{|S_A(X)|\ge\tfrac12\|S_A(X)\|_p\}
 \ge(1-2^{-p})^2
       \left(\frac{\|S_A(X)\|_p}{\|S_A(X)\|_{2p}}\right)^{2p}
 \ge e^{-C_dp}.
\]
Use the lower moment bound to obtain the second assertion.
The norm is nonzero for $A\ne0$, since the independent centered
families have positive variances and the distinct tensor monomials
are orthogonal in $L_2$.
\end{proof}

The same tail conclusions hold under Theorem~\ref{thm:regular}, with
constants depending on $d,K$. Indeed,
\cite[Lemma 3.10]{Mel19} gives
$\|b\|_{X,ap}\le C_Ka^m\|b\|_{X,p}$ for $a\ge1$.
Rowwise Euclidean duality transfers this support-function comparison
to the convex hulls of the block bodies. Applying it once per block yields
\[
 \|A\|_{ap}^N\le C_{d,K}a^{md}\|A\|_p^N.
\]
Theorem~\ref{thm:regular} therefore gives bounded doubling of the chaos
moments, and the same Markov and Paley--Zygmund argument applies.

\subsection{Returning to one family of variables}
Suppose now that $X_1,\ldots,X_n$ are independent normalized symmetric
variables with log-concave tails, and $A$ is invariant under permutations
of its $d$ indices and vanishes whenever two indices coincide.
The decoupling inequality of de la Pe\~na and Montgomery-Smith
\cite{dPMS95}, in the form stated in \cite[Theorem 1.1]{AL12},
compares
\[
 \left\|\sum_{\mathbf i}a_{\mathbf i}X_{i_1}\cdots X_{i_d}\right\|_p
 \quad\text{and}\quad
 \left\|\sum_{\mathbf i}a_{\mathbf i}
                 X_{i_1}^1\cdots X_{i_d}^d\right\|_p
\]
in both directions with constants depending only on $d$,
where the families on the right are independent copies.
Consequently \eqref{eq:general-natural-moment-formula} and
\eqref{eq:tail-bounds} hold also for these nondecoupled tetrahedral
chaoses, with $N_i^j=N_i$ and adjusted order-dependent constants.
The same decoupling argument applies under
Theorem~\ref{thm:regular}, with constants depending on $d,K$.
The diagonal-vanishing assumption is needed for this decoupling
consequence; the decoupled theorems themselves allow arbitrary arrays.

\end{document}